\documentclass{amsart}

\usepackage{bm}
\usepackage{amsmath}
\usepackage{amssymb}
\usepackage[all]{xy}
\newtheorem{thm}{Theorem}[section]
\newtheorem{cor}[thm]{Corollary}
\newtheorem{lem}[thm]{Lemma}

\newtheorem{prp}[thm]{Property}

\theoremstyle{definition}
\newtheorem{defn}[thm]{Definition}

\theoremstyle{remark}
\newtheorem{rmk}[thm]{Remark}
\numberwithin{equation}{section}
\begin{document}

\title{Notes on $A_\infty$-algebra and its endomorphism I}
\author{ Jie \ Zhao }
\address{Jie Zhao, University of Wisconsin-Madison, Madison, WI 53705, USA}
\email{jzhao@math.wisc.edu}

\begin{abstract}
In this notes, we study some basic results about deformation of $A_\infty$ structure. First, we point out
the $A_\infty$ structure itself is always homologous to zero in its Hochschild cohomology, and in contrast in the strict unital case the unit element is always nonzero if the algebra is supported in non-negative degrees. Then we study a natural two-dimensional rescaling of $A_\infty$ structure. We prove that such rescaled $A_\infty$ structure is always quasi-isomorphic to the original one through an explicit strict endomorphism. At last, we study the curved $A_\infty$ structure case, and show the deformation with Maurer-Cartan element or bounding cochain as in the Lagrangian Floer Homology theory can be naturally derived from a weakly strict endomorphism.

\end{abstract}
\maketitle

Given an $A_\infty$ algebra $(A,m)$, we are always interested in certain deformation of the $A_\infty$ structure. From the point of view of deformation theory, formally we need to solve the Maurer-Cartan typer equation. Generally, it is not easy to get explicit solutions and then write down the new $A_\infty$ structure. Here we consider some special cases of explicit deformation. We show that such deformation always can be derived from some (weakly) strict endomorphism equivalently.

As a quick review, first we recall some basical definition and properties about $A_\infty$ algebra and its morphism. We adapt the sign convention for morphism here and make some explicit calculation of the Hochschild cohomology. Using the \textit{euler derivation element} introduced by Seidel in his notes \cite{21}, we are able to show that $[m]=0 \in \text{HH}^2(A,A)$ in both strict and curved case. In contrast, in the unital and strict case, by degree argument, we always have $[\textbf{e}] \ne 0 \in \text{HH}^{0}(A,A)$ if the algebra is supported in non-negative degrees.

Then we study two types of deformation. The first deformation comes from the rescaling of the $A_\infty$ structure. From the construction equation of the $A_\infty$ structure, it turns out there is a natural two-dimensional family of rescaling. Such deformation can be proved to be trivial up to a quasi-isomorphism. In fact, we can explicitly construct such quasi-isomorphism, either by the \textit{euler derivation element} or the \textit{pseudo-isotopy} method developed by Fukaya \cite{4}. The \textit{pseudo-isotopy} method provides an explicit way to integrate the infinitesimal deformation along a path into a real deformation represented by an endomorphism. 

The second deformation is the one we used in the Lagrangian Floer Homology theory. From the study of homomorphism in the curved $A_\infty$ case, we show that the Maurer-Cartan equation and thus its induced deformation of $A_\infty$ structure could be naturally derived from the construction equation of some weakly strict endomorphism, which is called \textit{an almost identity endomorphism} in the later part. In this way, we are able to get a general condition for the \textit{Maurer-Cartan element} or \textit{bounding cochain} to deform a curved $A_\infty$ structure into a simpler one to achieve the cohomology theory.\\

\newpage

\section{Strict $A_\infty$ structure}

First let us follow Keller \cite{12} and make a quick review of strict $A_\infty$ algebra. Notice that here we adapt the sign system of \cite{17} for morphisms instead, which will be better for later calculation of Hochschild cohomology. As usual, our ground field is degree free. 

\begin{defn}
A strict or non-curved $A_\infty$-algebra $(A, m)$ over filed \textbf{F} is a $\mathbf{Z}$-graded vector space: 
$
A= \bigoplus_{k \in Z} A^k 
$
endowed with graded $\textbf{F}$-linear maps
$
m_n: A^{\otimes n} \rightarrow A, n \geq 1
$
of degree $(2-n)$ satisfying the construction equations:
\begin{displaymath}
\sum_{n = r+s+t} (-1)^{rs+t} m_u (\mathbf{I}^{\otimes r}\otimes m_s \otimes \mathbf{I}^{\otimes t})=0          
\end{displaymath}
\end{defn}

In the strict case, we have the index $u = r+1+t$ with $r,t \geq 0$ and $s \geq 1$. Thus we always have $u \geq 1$.
The first equation comes as: $m_1^2 =0$ for $n=1$.

Notice that when these formulas are applied to elements, additional signs appear because of the Koszul
sign rule:
\begin{displaymath}
(f \otimes g)(x \otimes y)= (-1)^{|g||x|}f(x)\otimes g(y)
\end{displaymath}
where $f, g$ are graded maps, $x,y$ are homogeneous elements and the $|g|, |x|$ denote
their degrees.

\begin{defn}
A homomorphism of strict $A_\infty$-algebras $f: (A,m^A) \rightarrow (B,m^B)$ is a family of graded maps
$
f_n: A^{\otimes n} \rightarrow B
$
of degree $(1-n)$ satisfying the compatible equations:
\begin{displaymath}
\sum_{n = r+s+t} (-1)^{rs+t} f_u (\mathbf{I}^{\otimes r}\otimes m^A_s \otimes \mathbf{I}^{\otimes t})=   
\sum_{n = i_1 + \cdots i_k} (-1)^{sign} m^B_k (f_{i_k} \otimes \cdots \otimes f_{i_1})       
\end{displaymath}
here the sign on the right hand side is given by:
\begin{displaymath}
sign = (k-1)(i_1-1)+(k-2)(i_2-1)+ \cdots + 2(i_{k-2}-1)+(i_{k-1}-1)
\end{displaymath}
If in addition, $f_1$ induces an isomorphism from $H^*(A,m^A_1)$ to $H^*(B,m^B_1)$, then the homomorphism $f$ is called a quasi-isomorphism.
\end{defn}

\begin{defn}
An $A_\infty$ homomorphism $f: (A,m^A) \rightarrow (B,m^B)$ is called a strict homomorphism if $f_n = 0$ for all $n \geq 2$. In the $A=B$ case, we call it a strict endomorphism.
\end{defn}
Notice that in this case, the equation for the homomorphism maps can be simplified to:
\begin{displaymath}
f_1 m^A_n = m^B_n (f_1 \otimes f_1 \otimes \cdots \otimes f_1)
\end{displaymath}

Moreover, if we have $A=B$ and the strict endomorphism $f: (A,\tilde{m}) \rightarrow (A,m)$ has the nonzero component $f_1$ an automorphism on $A$. Then in fact we will end with:
\begin{displaymath}
\tilde{m}_n = f_1^{-1} \circ m_n (f_1 \otimes f_1 \otimes \cdots \otimes f_1)
\end{displaymath}
Then $\tilde{m}$ is just the pullback $A_\infty$ structure on $A$ by endomorphism $f$. Thus the study of strict homomorphism becomes much easier than the general cases. In fact, in some geometrical situations, study of such homomorphism already gives us some important information of the structure.

\begin{defn} A strict $A_\infty$-algegra $(A,m^A)$ is isomorphic to $(B,m^B)$, if there exists $A_\infty$ homomorphism $f: A \rightarrow B$
and $g: B \rightarrow A$ such that $f\circ g$ and $g\circ f$ are both homotopy to identity.
\end{defn}

As a specialty for $A_\infty$ structure, we have the following \textit{Whitehead theorem} about quasi-isomorphism and isomorphism.
\begin{thm} For two strict $A_\infty$-algebra $A$ and $B$, if there exists a quasi-isomorphism $f: A \rightarrow B$, then $A$ and $B$ are isomorphic.\\
\end{thm}

Next, let us recall the construction of Hochschild cohomology of $A_\infty$ algebra, which is important for the study of the infinitesimal deformation of the $A_\infty$ structure. Here we use the description given in \cite{19}.

Consider a Lie algebra structure on the bigraded vector space:
\begin{displaymath}
C^{n,m}(A,A) = \text{Hom}^m(A^{\otimes n}, A) = \prod_i \text{Hom}( (A^{\otimes n})^i, A^{i+m} )
\end{displaymath}
For $f \in C^{n,k}(A,A)$ and $g \in  C^{m,l}(A,A)$, the Lie bracket or Gerstenhaber bracket $[f,g] \in C^{n+m-1, l+k}$ is given by the following:
\begin{displaymath}
[f,g] = \sum_{i=0}^{n-1}(-1)^{\delta_1} f(\mathbf{I}^{\otimes i} \otimes g \otimes \mathbf{I}^{\otimes n-i-1})- (-1)^{(n+k-1)(m+l-1)} \sum_{i=0}^{m-1}(-1)^{\delta_2} g(\mathbf{I}^{\otimes i} \otimes f \otimes \mathbf{I}^{\otimes m-i-1})
\end{displaymath}
here:
\begin{displaymath}
\delta_1 = (n-1)(m-1)+(n-1)l+i(m-1), \quad \delta_2 = (m-1)(n-1)+(m-1)k+i(n-1)
\end{displaymath}
Notice that in the formula if we have $n=0$ or $m=0$, then the corresponding term should take $0$ value. For example, if $f \in C^{n,k}(A,A)$ and $g \in C^{0,l}$, then:
\begin{displaymath}
[f,g]= \sum_{i=0}^{n-1}(-1)^{\delta_1} f(\mathbf{I}^{\otimes i} \otimes g \otimes \mathbf{I}^{\otimes n-i-1}) 
\end{displaymath}
Also from the formula, we have $[f,g] = -(-1)^{(n+k-1)(m+l-1)}[g,f]$. Thus when the total degree of $f$ or (n+k) is odd, we always have $[f,f]=0$.\\

For $A_\infty$ structure, we always have the componet map $m_k \in C^{k, 2-k}$, with total degree 2. Take the notation $m= \sum m_n$, by definition we have the following facts: 

1)  m $\in C^{*,*}(A,A)$ of total degree 2,

2) [m,  m]=0,

3) D = [m, -] is a differential on $C^{*,*}(A,A)$.

\begin{rmk}
Here fact 3) comes from the Jacobi identity of the Lie bracket and also the fact 1) and 2).
\end{rmk}

\begin{defn}
Let $(A,m)$ be a strict $A_\infty$ algebra, then its Hochschild cohomology is defined as:
\begin{displaymath}
\text{HH}^*(A,A) = \text{H}^*(\prod_i C^{i, *-i}(A,A), D)
\end{displaymath}
\end{defn}
For example, we always have: $[m] \in \text{HH}^2(A,A)$. In fact, we can show it is always homologous to zero in the Hochschild cohomology.

\begin{defn}
Given an $A_\infty$-algebra $(A,m)$, we call the element $\mathbf{E} \in C^{1,0}(A,A)$ its \textit{euler derivation element}, which is given by:
\begin{displaymath}
\mathbf{E}(\alpha) = deg(\alpha) \cdot \alpha, \ \forall \alpha \in A
\end{displaymath}
\end{defn}

\begin{thm}
For any strict $A_\infty$ algebra $(A,m)$, we always have the following fact: $[m]=0 \in \text{HH}^2(A,A)$.
\end{thm}

\begin{proof}

First let us consider the identity element $\mathbf{I}: A \rightarrow A$, by definition we will get $\mathbf{I} \in C^{1,0}(A,A)$ and $ [m_k, \mathbf{I}] = (k-1)m_k$, thus we have:
\begin{displaymath}
D(\mathbf{I}) = \sum_{k \geq 1} (k-1)m_k
\end{displaymath}

Then let us take the \textit{euler derivation element} $\mathbf{E} \in C^{1,0}(A,A)$. By definition, we will get another: 
\begin{displaymath}
D(\mathbf{E}) = \sum_{k \geq 1} (k-2)m_k
\end{displaymath}

Therefore, we have the following identity:
\begin{displaymath}
D(\mathbf{I}-\mathbf{E}) = \sum_{k \geq 1} m_k= m
\end{displaymath}

i.e. $[m]=0 \in \text{HH}^2(A,A)$.

\end{proof}

Notice that the above solution $(\mathbf{I}-\mathbf{E}) \in C^{1,0}(A,A)$. In fact, by degree reason, there might be other solutions for $D(a)=m$, for example with $a \in C^{0,1}(A,A)$. In a special case, if $\{m_k\}$ has infinite non-zero component elements, and moreover has an element $a \in C^{0,1}(A,A)$ with the \textit{divisor property} as follows:
\begin{displaymath}
m_1(a(\textbf{1}))=0, \quad \sum_{i=0}^{k-1} (-1)^i m_k(\mathbf{I}^{\otimes i} \otimes a(\textbf{1}) \otimes \mathbf{I}^{\otimes k-i-1})= m_{k-1} \ \text{for} \ k \geq 2
\end{displaymath}
Then $a \in C^{0,1}(A,A)$ will become another solution, i.e. $D(a)=m$.

As for other cocycles in the Hochschild cohomology, generally it is not easy to get explicit construction. But we still have some clue if we simply consider element in $C^{0,l}(A,A)$ for example in the unital case. 

\begin{cor}
Given a unital strict $A_\infty$ algebra $(A,m,\textbf{e})$ with unit element $\textbf{e}$, we always have $D(\textbf{e})=0$. Moreover, if the algebra is supported in non-negative degrees, then we will have $[\textbf{e}] \ne 0 \in \text{HH}^{0}(A,A)$, thus \text{HH}$^*(A,A)$ is non-vanishing.
\end{cor}
\begin{proof} The first part directly comes from the property of unit. Now let us consider the second part by contradiction. Suppose there is an element $f \in C^{m,n}$ such that $[m_k,f]$ contains $\textbf{e}$ term. Let us carefully check its degree and consider the degree equation, that is:
\begin{displaymath}
k+ m-1 =0, \quad 2-k+n =0
\end{displaymath}
Notice that we have the requirement of $k \geq 1$ and $m \geq 0$, thus the only possibility is $m=0, n=-1$, i.e. $f \in C^{0,-1}$ and $f(\textbf{1}) \in A^{-1}$. But here our vector space is non-negative graded. Contradiction.

\end{proof}

\begin{rmk}
Follow this consideration, in fact we will get a natural inclusion map:
\begin{displaymath}
i: \{ c \in C^{0,l}(A,A)\ | \ [m,c]=0  \} \hookrightarrow \{ c(\textbf{1}) \in A^l \ | \ m_1(c(\textbf{1}))=0  \}
\end{displaymath}
Thus in the strict case, the Hochschild cohomology will contain partial data of the $m_1$-cohomology of the original $A_\infty$ algebra, that is some special element with each cyclic sum of $m_k$ component vanishes.
\end{rmk}

\newpage

\section{rescaling of $A_\infty$ structure}

In this part, we consider a basic deformation of the $A_\infty$ structure given by rescaling. We still focus on
strict $A_\infty$ algebra case here. The discussion in the curved case will be given in the later part.

The rescaling here comes from the following observation of the construction equation of $A_\infty$ structure: in the construction equation, when $n$ is fixed, we have: $u+s = n+1$ as fixed value . Therefore, by rescaling of the form $\tilde{m}_k = \lambda^{f(k)} \cdot m_k $, as long as the function $f(k)$ satisfies the identity:
\begin{displaymath}
f(u)+f(s)=const, \quad \forall \ u+s=n+1
\end{displaymath}
The new maps $\{\tilde{m}_k\}$ give another $A_\infty$-strucutre, i.e. it still satisfies the construction equation.
Notice that functions with the above property must have the form of $f(k)=a\cdot k+b$ as a linear function. Thus we have the following:

\begin{thm}
If $(A, m) $ is an $A_\infty$ algebra, then by the rescaling of each component map as:
\begin{displaymath}
\tilde{m}_k = \lambda^{a\cdot k +b} \cdot m_k, \ \text{with} \ \lambda, a, b \in \textbf{F}
\end{displaymath}
we get another $A_\infty$ algebra $(A, \tilde{m} )$.
\end{thm}

\begin{rmk} Let us consider the set of all $A_\infty$ structure on the space $A$, i.e. $\mathcal{M}(A, m) = \{ \overline{m} \ |\ \overline{m} \ \text{is any $A_\infty$ structure on A}\}$. If we equip the space $\mathcal{M}(A,m)$ with certain topology, then through the simple path: $m(t) := t \cdot m, t \in [0,1]$, we will get the space of all $A_\infty$ structure on $A$ is path connected and contractible. Thus it would be better to study $ \mathcal{M}(A,m)-\{  \textbf{0}  \} $ in stead.
\end{rmk}

Now let us consider the relation between the rescaled $A_\infty$ structure $\tilde{m}$ and the original one $m$. Notice that when $\lambda \ne 0$, the deformation given above essentially comes in two families.\\

\textit{type 1.} If we take $\lambda \ne 0, a=1, b=-1$, that is $$\tilde{m}_k = \lambda^{k-1} \cdot m_k, \quad \forall k \in \mathbf{N}$$ Then we have $(A, \tilde{m}) $ is quasi-isomorphic to $(A, m) $. Here we just need a strict homomorphism $f: (A, \tilde{m}) \rightarrow (A,  m) $ given by: $f_1 = \lambda $ and $f_n=0$ for all $n \ne 1$. Then it is easy to check it is a quasi-isomorphism.\\

\textit{type 2.} If we take $\lambda \ne 0, a=1, b=-2$, that is $$\tilde{m}_k = \lambda^{k-2} \cdot m_k, \quad \forall k \in \mathbf{N}$$ We still have $(A, \tilde{m}) $ is quasi-isomorphic to $(A, m) $. In this case, the strict quasi-isomorphism  is induced by the \textit{euler derivation element}, that is $f_1 = \lambda^{\mathbf{E}}$ given by:
\begin{displaymath}
f_1(\alpha) = \lambda^{\mathbf{E}}(\alpha):= \lambda^{deg(\alpha)} \cdot \alpha, \ \forall \alpha \in A
\end{displaymath}

General case is the combination of the two types, thus we have the result:
\begin{cor}
Given an $A_\infty$-algebra $(A,m)$, the rescaling of the $A_\infty$ structure given as above is always a quasi-isomorphism if $\lambda \ne 0$.
\end{cor}
\begin{proof}
Consider the rescaling given by:
\begin{displaymath}
\tilde{m}_k = \lambda^{a\cdot k +b} \cdot m_k, \ \text{with} \ \lambda, a, b \in \textbf{F}
\end{displaymath}
From the splitting: $a \cdot k +b = (2a+b)(k-1)+(-a-b)(k-2)$, we just need to take a strict quasi-isomorphism $f: (A, \tilde{m}) \rightarrow (A,  m) $ given by: $$f_1 = \lambda^{2a+b} \cdot (\lambda^{\mathbf{E}})^{-a-b}$$

\end{proof}

\begin{rmk}
In the DG-algebra case, the \textit{type 1} deformation gives the new algebra structure as:
$\tilde{m}_1=m_1, \tilde{m}_2= \lambda m_2$, while the \textit{type 2} deformation gives the new algebra structure as:
$\tilde{m}_1=\lambda^{-1} m_1, \tilde{m}_2= m_2$. As $A_\infty$-algebra, they are both quasi-isomorphic to the original one.
\end{rmk}

The construction of the strict quasi-isomorphism in the \textit{type 2} case indicates the $f_1$ map might be recovered from the relation $D(\mathbf{I}-\mathbf{E})=m $ in the Hochschild cohomology by certain integration of the infinitesimal deformation. Such integration indeed can be achieved, for example by the pseudo-isotopy method developed by Fukaya \cite{4}. Detailed description and recently application of pseudo-isotopy could be found in FOOO's book \cite{5} and Tu's paper \cite{23}. Here we recall the construction as given in \cite{23}.

Given a formal based path $l(t)$ on $\mathcal{M}(A,m)$ with $l(0)=m$, let us denote the $A_\infty$ structure along the path by $m^t = l(t) \in \mathcal{M}(A,m)$, and $\delta^t = m^t - m^0=m^t-m$.

\begin{defn} The based path $l(t)$ is called a pseudo-isotopy if there exists an element $h^t \in Coder(T(A)) \otimes C^{\infty}[0,1]$ with total degree 1, satisfying the Maurer-Cartan equation:
\begin{displaymath}
\frac{d \delta^t}{d t}= [m^t, h^t]
\end{displaymath}

\end{defn}

For example, if we just consider the simple linear path as $l(t) = m + t \cdot \delta$, then one necessary condition to make it into a pseudo-isotopy will be:
\begin{displaymath}
[m, \delta]=0, \quad \text{and} \quad [\delta, \delta]=0.
\end{displaymath}

\begin{thm}[\cite{4}] If a based path $l(t)$ is a pseudo-isotopy, then it induces a formal endomorphism $f: (A,m^1) \rightarrow (A,m^0)$. If the endomorphism $f$ converges, then it becomes a quasi-isomorphism.
\end{thm}

Now let us sketch the construction of the component maps $\{f_k \}$ of the formal $A_\infty$ endomorphism. 

For any non-negative integer $k$, consider a ribbon tree $T$ with $k$ leaves. Le us denote its set of edges by $E(T)$, and the set of vertices by $V(T)$. Then $V(T)$ is the disjoint union of $V^{int}(T)$ consisting of interior vertices and $V^{ext}(T)$ consisting of exterior vertices. The exterior vertices are necessarily of valency 1, and the interior vertices may have arbitrary positive valency. Let us denote the set of all such ribbon trees by $Gr(k)$.

First, take a partial ordering on the set $V^{int}(T)$, we say $v_1 \leq v_2$ if $v_2$ is contained in the
shortest path from $v_1$ to the root. Given $t \in [0,1]$, a time ordering on $T$ bounded by $t$ is a non-decreasing
map $\tau: V^{int}(T) \rightarrow [0, t]$. Take $M^t(T)$ as the set of all time orderings on
$T$, and endow it with the subset topology of $[0, t]^{|V^{int}(T)|}$. The Euclidean measure on $[0, t]^{|V^{int}(T)|}$
induces a measure on $M^t(T)$, which we denote by $d\tau$.

Second, for each pair $(T,\tau)$, we define the kernel map $c (T,\tau): A^{\otimes k} \rightarrow
A$ in the following way. On each internal vertex $v \in V^
{int}(T)$ we put the linear map $
h^{\tau(v)}_{
val(v)-1}$ where $val(v)$ is the valency of the vertex v. Then the map $c (T, \tau)$
is simply the “operadic” composition of these linear maps according to the tree
$T$.

Then for $t \in [0,1]$, we have the formal endomorphism of the $A_\infty$ structure along the path $f^t: (A,m^t) \rightarrow (A,m^0)$, with the components $\{f^t_k\}$given by:
\begin{displaymath}
f^t_k = \sum_{T \in Gr(k)} c^t(T) 
\end{displaymath}
where 
\begin{displaymath}
\quad c^t(T)=\mathbf{I}, \quad \text{if} \quad |V^{int}(T)|=0
\end{displaymath}
and
\begin{displaymath}
\quad c^t(T)=\int_{M^t(T)} c(T,\tau) d \tau, \quad \text{if} \quad |V^{int}(T)| \geq 1
\end{displaymath}

Take $t=1$, then we get the $A_\infty$ endomorphism $f$ as in the theorem, which is just the $f^1: (A,m^1) \rightarrow (A,m^0)$, thus given by:
\begin{displaymath}
f_k = \sum_{T \in Gr(k)} c^1(T), \qquad \text{where} \quad c^1(T)=\int_{M^1(T)} c(T,\tau) d \tau
\end{displaymath}
 
\begin{cor} If along the pseudo-isotopy $l(t)$, the solution of Maurer-Cartan equation $h^t$ has just one component as $h^t = h^t_1 \in C^{1,0}(A,A)$ for $t \in [0,1]$, then for $t \in [0,1]$ the induced formal endomorphism $f^t: (A,m^t) \rightarrow (A,m^0)$ is always a strict endomorphism, which is given by:
\begin{displaymath}
f^t_1 = \exp(C(t)), \quad \text{with} \quad
C(t) = \int_0^t h^{\tau}_1 d \tau
\end{displaymath}
Especially, the formal endomorphism $f: (A,m^1) \rightarrow (A,m^0)$ is a strict one, which is given by:
\begin{displaymath}
f_1 = \exp(C(1)), \quad \text{with} \quad
C(1) = \int_0^1 h^{\tau}_1 d \tau
\end{displaymath}
Moreover, if  we have the term $C(t)$ converges, then $f^t$ becomes a quasi-isomorphism of $A_\infty$ structure.
\end{cor}

\begin{proof} By the construction of the formal endomorphism given above, in this case we have $k=1$, and only $f^t_1$ is nonzero, thus $f^t$ is always a strict endomorphism. 

In addition, for each $T \in Gr(1)$, we have the calculation:

1) if $|V^{int}(T)|=0$, then: $c^t(T)=\mathbf{I}$,

2) if $|V^{int}(T)| \geq 1$, set $n= |V^{int}(T)|$, follow the tree type, we get:
\begin{displaymath}
c^t(T) = \int_0^t \int_0^{t_{n}} \cdots \int_0^{t_{2}} (h_1^{t_{n}} h_1^{t_{n-1}} \cdots h_1^{t_1}) \ dt_1 dt_2 \cdots dt_n = \frac{1}{n !} (C(t))^n
\end{displaymath}
Therefore, we have:
\begin{displaymath}
f^t_1 = \sum_{T \in Gr(1)} c^t(T) = \sum_{n=0}^{\infty} \frac{1}{n !}(C(t))^n = \exp(C(t))
\end{displaymath}

The rest is directly from this expression.
\end{proof}

\begin{rmk}
This corollary give the formula to integrate the infinitesimal deformation into a global deformation of the $A_\infty$ structure, which is given in the form of a strict endomorphism. In this sense, strict
endomorphism could be viewed equivalently as a simple type of deformation.

\end{rmk}

\textbf{Example.1} As a simple example, let us recover the strict quasi-isomorphism in the case: $a=0, b=1, \lambda =2$, that is the case $\tilde{m} = 2 m$.

In this case, we take a simple path: $l(t)=(1+t)m$, then $\delta^t = t \cdot m$, and the Maurer-Cartan equation
$\frac{d \delta^t}{dt} = [m^t, h^t]$ is simplified to:
\begin{displaymath}
[(1+t)m, h^t] = m
\end{displaymath}
the solution exists since $[m,\mathbf{I}-\mathbf{E}]=m$, and we get: $h^t = \frac{1}{1+t} (\mathbf{I}-\mathbf{E}) \in C^{1,0}(A,A)$, thus we get the term:
\begin{displaymath}
C(1) = \int_0^1 h^{\tau} d\tau = \ln 2 \cdot (\mathbf{I}-\mathbf{E})
\end{displaymath}
therefore, by the corollary, one of the strict quasi-isomorphism $f: (A,2m) \rightarrow (A,m)$ could be given by:
\begin{displaymath}
f_1 = \exp [\ln 2 \cdot (\mathbf{I}-\mathbf{E})] = 2^{\mathbf{I}-\mathbf{E}}= 2 \cdot (2^{\mathbf{E}})^{-1}
\end{displaymath}
which fits the strict quasi-isomorphism given in Corollary 2.3.\\

\textbf{Example.2} Similarly, if we take the simple path $l(t)=(1-t)m$, then it is a path from $m$ to \textbf{0}. In this case, we have $\delta^t = -t \cdot m$, and $h^t = \frac{1}{1-t}\cdot (\mathbf{I}-\mathbf{E})$, which is not well defined at $t=1$. Moreover in the integral we still get the divergence:
\begin{displaymath}
C(1) = \int_0^1 h^{\tau} d\tau = \int_0^1 \frac{1}{1-t} dt \cdot (\mathbf{I}-\mathbf{E})
\end{displaymath}
thus we cannot use the pseudo-isotopy to detect the quasi-isomorphism relation in this way.\\

As a summary, we state an alternative proof of Corollary 2.3 here.
\begin{proof}
In the general case, given the rescaling of $A_\infty$ structure of $(A,m)$: 
\begin{displaymath}
\tilde{m}_k = \lambda^{a\cdot k +b} \cdot m_k, \quad \lambda \ne 0
\end{displaymath}

Let us consider the following simple path:
\begin{displaymath}
l(t) = \tilde{m}^t, \quad \text{with}\ \tilde{m}^t  \ \text{given by:} \quad \tilde{m}^t_k = [1+t\cdot (\lambda -1)]^{a \cdot k +b} \cdot m_k
\end{displaymath}

From the the Maurer-Cartan equation
$\frac{d \delta^t}{dt} = [m^t, h^t]$, we get one solution:
\begin{displaymath}
h^t = \frac{\lambda -1}{1+t(\lambda -1)} [(2a+b)\mathbf{I} - (a+b)\mathbf{E}]
\end{displaymath}

Thus by Corollrary 2.7, we recover the quasi-isomorphism:
\begin{displaymath}
f_1 = \exp(\int_0^1 h^\tau d \tau) = \lambda^{2a+b} \cdot (\lambda^{\mathbf{E}})^{-a-b}
\end{displaymath}

\end{proof}

\begin{rmk}
From the construction of the endomorphism, we can see since the existence of $h^t_1 \in C^{1,0}(A,A)$ and the space $Gr(1)$ has infinite elements, exponential type term generally will appear during the calculation.
\end{rmk}

\newpage

\section{Curved $A_\infty$ structure and its endomorphism}

In this part, we carefully check the definitions for curved $A_\infty$ structure and then study its homomorphism. We omit the convergence argument about infinite series here. Details about the \textit{T-adic} topology on $A_\infty$ structure can be found in FOOO's book \cite{5}. For simplicity, let us just assume $m_k =0$ for $k \gg 1$ in this part. \\

In this part, we use the sign system for evaluation at elements as in \cite{5}. That is better for explicit calculation of the deformation. As before, here our ground field is still degree free.

\begin{defn}
A curved $A_\infty$-algebra $(A, m)$ over filed \textbf{F} is a $\mathbf{Z}$-graded vector space: 
$
A= \bigoplus_{k \in Z} A^k 
$
endowed with graded $\textbf{F}$-linear maps
$
m_n: A^{\otimes n} \rightarrow A, n \geq 0
$
of degree $(2-n)$ satisfying the construction equations:
\begin{displaymath}
\sum_{n = r+s+t} (-1)^{\delta(x_1, \cdots, x_r)} m_u (x_1, \cdots, x_r, m_s(x_{r+1}, \cdots, x_{r+s}), x_{r+s+1},\cdots, x_n)=0          
\end{displaymath}
\end{defn}
Here the index $u = r+1+t$ with $r,t \geq 0$ and $s \geq 0$. When $s=0$, we 
consider $m_0$ as $m_0(\textbf{1})$ in the definition.
In this case, we have the sign convention: 
\begin{displaymath}
\delta(x_1, \cdots, x_r) = \sum_{i=1}^{r} (|x_i| - 1)
\end{displaymath}

Now the first two construction equations change into: 

$m_1(m_0(\textbf{1})) =0$ for input number $n=0$,

$ m_1 (m_1(x)) + m_2(m_0(\textbf{1}),x) +(-1)^{|x|-1} m_2(x, m_0(\textbf{1}))=0$
for input number $n=1$.

Thus we cannot expect $m^2_1=0$ anymore unless:
\begin{displaymath}
m_2(m_0(\textbf{1}),x) = (-1)^{|x|}m_2( x, m_0(\textbf{1}) ),\ \text{for all}\ x\in A
\end{displaymath}

\begin{defn}
A homomorphism of curved $A_\infty$-algebras $f: (A,m^A) \rightarrow (B,m^B)$ is a family of graded maps
$
f_n: A^{\otimes n} \rightarrow B
$
of degree $(1-n)$ satisfying the compatible equations:
\begin{align*}
&\sum_{n = r+s+t} (-1)^{\delta(x_1, \cdots, x_r)} f_u (x_1, \cdots, x_r, m^A_s(x_{r+1}, \cdots, x_{r+s}), x_{r+s+1},\cdots, x_n)\\
= & \sum_{n = i_1 + \cdots i_k}  m^B_k (f_{i_1}(x_1, \cdots, x_{i_1}), \cdots, f_{i_k}(x_{n-i_k+1},\cdots, x_n))       
\end{align*}
here the sign convention term $\delta$ is the same as in the previous definition.
\end{defn}
Notice here on the left side we still take index $u = r+1+t$ with $r,t \geq 0$ and $s \geq 0$, and on the right side we should take \textbf{$k \geq 1$ if $n \geq 1$ but $k \geq 0$ if $n = 0$}, and we consider $f_0$ as 
$f_0(\textbf{1}) \in A$ in the equation, which is a degree 1 term.

\begin{defn}
We call a curved $A_\infty$ homomorphism $f: (A,m^A) \rightarrow (B,m^B)$ is a strict homomorphism if $f_0 = 0$, and weakly strict homomorphism if $f_n =0$ for all $n \geq 2$. In addition, if $A=B$ and $f_1 = \mathbf{I}$ on $A$, then we call it an almost identity endomorphism; if $f_1 =  \mathbf{I}$ and $f_0=0$, then we call it an identity automorphism.
\end{defn}

Now let us check the results of previous sections in the curved case. First, it is easy to check the definition of Hochschild cohomologty as in Section.1 can be directly generalized to the curved case. Moreover, we still have the identity: $[m, \mathbf{I}-\mathbf{E}] = m$, thus the following result is still valid:
\begin{thm}
For any curved $A_\infty$ algebra $(A,m)$, we have $[m]=0 \in HH^2(A,A)$, in fact $ m= [m, \mathbf{I}-\mathbf{E}] $. 
\end{thm}
Similarly, the rescaling deformation construction and also the corresponding strict endomorphism are still valid for the curve case.\\

Now let us focus on explicit calculation of the homomorphism in the curved case. We adapt some notation in the Lagrangian Floer Homology theory. Under our sign system, each term comes with the positive sign.
\begin{align*}
 & m(e^b) := m_0(\textbf{1})+ m_1(b)+ m_2(b,b) + \cdots + m_k(b,\cdots,b)+ \cdots\\
 m(e^b x_1  e^b x_2 & e^b \cdots e^b x_k e^b) := \sum m_l(b,\cdots,b,x_1,b,\cdots,b, x_2,b, \cdots,b, x_k, b,\cdots,b )
\end{align*}

Here the index $l \geq k$, and the total times of $b$ input is $l-k$. We sum over all the combinatorial type of inserting $b$ in between of $\{ x_i \}$ for all possible $l$.

\begin{lem}
If a curved $A_\infty$ endomorphism $f:(A,\tilde{m}) \rightarrow (A,m)$ is an almost identity endomorphism, i.e. $f_1 = \mathbf{I}$ and $f_n =0$ for all $n \geq 2$, then:

1. \quad $ \tilde{m}_0(\textbf{1}) = m(e^{f_0(\textbf{1})})$, 

2. \quad $\tilde{m}_n (x_1, \cdots, x_n) =  m(e^b x_1  e^b x_2 e^b \cdots e^b x_n e^b)$.

\end{lem}

\begin{proof} By definition of curved $A_\infty$ homomorphism, we have the equation:
\begin{align*}
&\sum_{n = r+s+t} (-1)^{\delta(x_1, \cdots, x_r)} f_u (x_1, \cdots, x_r, \tilde{m}_s(x_{r+1}, \cdots, x_{r+s}), x_{r+s+1},\cdots, x_n)\\
=  & \sum_{n = i_1 + \cdots i_k}  m_k (f_{i_1}(x_1, \cdots, x_{i_1}), \cdots, f_{i_k}(x_{n-i_k+1},\cdots, x_n))       
\end{align*}
with the sign convention term: $\delta(x_1, \cdots, x_r) = \sum_{i=1}^r (|x_i|-1)$.\\

For $n=0$ case, we have:
\begin{align*}
f_1(\tilde{m}_0(\textbf{1})) = m(e^{f_0(\textbf{1})})
\end{align*}

Since $f_1 = \mathbf{I}$, we get the first equation: $ \tilde{m}_0(\textbf{1}) = m(e^{f_0(\textbf{1})})$.\\

For general $n \geq 1$ case, notice on the left side $u \geq 1$ and $f_k=0$ for all $k \geq 2$, thus only one term exist on the left, that is: 
\begin{displaymath}
L.H.S. = f_1 (\tilde{m}_n(x_1, \cdots, x_n)) = \tilde{m}_n(x_1, \cdots, x_n)
\end{displaymath}

On the right side, just the terms which only contains $f_0(\text{1})$ and $f_1$ components will survive, thus:
\begin{align*}
R.H.S. &= m(e^{f_0(\textbf{1})} f_1(x_1)  e^{f_0(\textbf{1})} f_1(x_2) e^{f_0(\textbf{1})} \cdots e^{f_0(\textbf{1})} f_1(x_n) e^{f_0(\textbf{1})}) \\
&=m(e^{f_0(\textbf{1})} x_1  e^{f_0(\textbf{1})} x_2 e^{f_0(\textbf{1})} \cdots e^{f_0(\textbf{1})} x_n e^{f_0(\textbf{1})})
\end{align*}

Therefore, we get the second equation. Notice that follow the sign convention, no sign complicity happens in the identity above.
\end{proof}

Based on this result, now let us make some definitions.

\begin{defn}
Given a curved $A_\infty$ structure $(A,m)$, we take the following equation as the \textbf{Maurer-Cartan equation}:
\begin{displaymath}
m(e^b) = a,\quad a,b \in A= \bigoplus_{k \in Z} A^k
\end{displaymath}

We call an element $a \in A^2$ is invertible if the Maurer-Cartan equation $m(e^b)=a$ has
a solution $b \in A^1$. We call such $b$ is a Maurer-Cartan element associated to $a$.
\end{defn}
\begin{defn}
Given a curved $A_\infty$ structure $(A,m)$, we call the following set its invertible set:
\begin{displaymath}
I(A,m) = \{ a \in A^2 \ | \ a \ \text{is invertible}\}
\end{displaymath}
For $a \in I(A,m) $, then we call the following set its Maurer-Cartan set:
\begin{displaymath}
MC(a) = \{ b \in A^1 \ | m(e^b) =a \}
\end{displaymath}
\end{defn}

Notice that $I(A,m)$ is always a non-empty set, since $0 \in MC(a)$ for $a=m_0(\textbf{1})$. In fact, by direct rescaling, we can see $I(A,m)$ is always a path-connected set.\\

According to the lemma above, we have the following statement:
\begin{thm} [\cite{5}]
Given a curved $A_\infty$ algebra $(A,m)$ and $a \in A^2$, it can be deformed into another 
$A_\infty$ algebra $(A,\tilde{m})$ with $\tilde{m}_0(\textbf{1}) = a$ through an almost identity endomomorphism
if and only if $a \in I(A,m)$. 

Moreover if $a \in I(A,m)$, then each component $m_n$ can be deformed to:
\begin{align*}
\tilde{m}_n (x_1, \cdots, x_n) =  m(e^b x_1  e^b x_2 e^b \cdots e^b x_n e^b),  \ \text{for all} \ n \geq 0
\end{align*}
with element  $b \in MC(a)$. 
The almost identity endomorphism $f:(A, \tilde{m}) \rightarrow (A,m)$ is simply given by:
\begin{displaymath}
f_0 = b,\  \text{and} \ f_1 = \mathbf{I}
\end{displaymath}

\end{thm}

\begin{rmk} It is easy to check that for any $b \in A^1$, then deformed maps $\{\tilde{m}_k\}$ as defined above always satisfy the construction equation for $A_\infty$ algebra, i.e. $\tilde{m}$ is always an $A_\infty$ structure. Thus the Maurer-Cartan equation $\tilde{m}_0(\textbf{1}) = m(e^b)=a$ is purely the initial term condition for such deformation.

\end{rmk}

Now let us check the inverse endomorphism of such deformation and solve the inverse Maurer-Cartan equation. The following statement is an enhancement of the description in FOOO's book \cite{5}.
\begin{cor}
Given a curved $A_\infty$ algebra $(A,m)$ and a deformation given by an almost identity endomorphism $f:(A,\tilde{m}) \rightarrow (A,m)$ as above by $f_0=b, f_1=\mathbf{I}$, then it admits an inverse deformation of the same type given by another almost endomorphism $g: (A,m) \rightarrow (A,\tilde{m}) $ such that $f \circ g$ and $g\circ f$ are identity automorphism.
\end{cor}

\begin{proof} By observation, we just need to take the inverse strict endomorphism:
\begin{displaymath}
 g_0=-b, \quad \text{and} \ g_1=\mathbf{I}.
\end{displaymath}

 We check the inverse Maurer-Cartan equation here, that is:
\begin{displaymath}
\tilde{m}(e^{-b}) = m_0(\textbf{1})
\end{displaymath}
Recall that we have:
\begin{align*}
&\tilde{m}_0(\textbf{1}) = m(e^b) = m_0(\textbf{1}) + m_1(b) + m_2(b,b) + \cdots\\
\tilde{m}_n (x_1, &\cdots, x_n) =  m(e^b x_1  e^b x_2 e^b \cdots e^b x_n e^b),  \ \text{for all} \ n \geq 0
\end{align*}
From the combinatorial identity:
\begin{displaymath}
\sum_{i=0}^{n} (-1)^i C_n^i=0,
\end{displaymath}
we get the vanishing of coefficient of term $m_n(b, \cdots, b)$ in $\tilde{m}(e^{-b})$, thus only one term left and that is the identity: $\tilde{m}(e^{-b}) = m_0(\textbf{1})$.

Similarly by combinatorial cancellation, we also have the identity:
\begin{align*}
m_n (x_1, \cdots, x_n) =  \tilde{m}(e^{-b} x_1  e^{-b} x_2 e^{-b} \cdots e^{-b} x_n e^{-b}),  \ \text{for all} \ n \geq 1
\end{align*}
Together with the previous identity we already have on hand:
\begin{align*}
\tilde{m}_n (x_1, \cdots, x_n) =  m(e^b x_1  e^b x_2 e^b \cdots e^b x_n e^b),  \ \text{for all} \ n \geq 1
\end{align*}
we precisely get $f\circ g$ and $g\circ f$ are both identity automorphism.
\end{proof}

\begin{rmk} We should be very careful here. Even in the enhanced sense, generally we still cannot expect the following homotopy relation to be true:
 \begin{displaymath}
\text{HH}^*(A,A; m) \cong \text{HH}^*(A,A; \tilde{m})
 \end{displaymath}
The reason is the following: if we start with $ m $ a strict one, and try the deformation by the almost identity endomorphism as above, then general $\tilde{m}$ will become a curved one; or if we start with the initial $m$ a curved one, but the equation $\tilde{m}_0 (\textbf{1})= m(e^b) = 0$ has a solution, then $\tilde{m}$ turns into a strict one. In either case, we will have one strict and one curved $A_\infty$ algebra. However, under curtain conditions, as explored in \cite{2} and \cite{15}, we have the Hochschild cohomology of curved one always vanishing, but general the strict one is not. For example, in the unital and curved case, if the $A_\infty$ algebra is supported in non-negative degrees, the unit element $[\textbf{e}]$ is not necessary nonzero, but it is always nonzero in the strict case.

\end{rmk}

Now let us knock at the following problem: how to deform a curved $A_\infty$ structure $m$ into a simpler one $\tilde{m}$ in the sense that $\tilde{m}_1^2 =0$? Here we just focus on the simple case when such deformation can be achieved through a strict endomorphism, or even an almost identity endomorphism.

\begin{defn}
Given a curved $A_\infty$ algebre $(A,m)$, the center of $m_2$ is given by:
\begin{displaymath}
Z(m_2) = \{ c \in A^{even} \ | \  m_2(c,x) = (-1)^{|x|}m_2(x,c) \}
\end{displaymath}
Moreover, if an element $c \in Z(m_2)$ satisfies further identity:
\begin{displaymath}
m_k(x_1, x_2, \cdots,x_i, c, x_{i+1}, \cdots, x_{k-1}) \equiv 0 \quad \text{for all}\ k \geq 3, \ i \geq 0
\end{displaymath}
we call it has the partial unital property in the algebra $(A,m)$. Especially, if an element $\textbf{e} \in A^0$ has the partial unital property and satisfies the identity: 
\begin{displaymath}
m_1(\textbf{e})=0, \quad m_2(\textbf{e},x) \equiv  x
\end{displaymath}
then we call it the unit of the algebra $(A,m)$.
 
\end{defn}
 
Recall the construction equation: 
\begin{displaymath}
\tilde{m}_1^2(x) = (-1)^{|x|}\tilde{m}_2( x, \tilde{m}_0(\textbf{1}) ) - \tilde{m}_2(\tilde{m}_0(\textbf{1}),x)
\end{displaymath}
Follow the above notation, to achieve $\tilde{m}_1^2 \equiv 0$ is equivalent to make $\tilde{m}_0(\textbf{1}) \in Z ( \tilde{m}_2 )$. Return to the original structure $(A,m)$, we will arrive at:
\begin{cor} [\cite{5}]
Given a curved $A_\infty$ algebra $(A,m)$, it can be deformed into another 
$A_\infty$ algebra $(A,\tilde{m})$ with $\tilde{m}^2_1= 0$ through an almost identity endomorphism, if and only if
the following \textbf{obstruction equation}:
\begin{displaymath}
m(e^b m(e^b) e^b x e^b ) = (-1)^{|x|} m(e^b x e^b m(e^b)  e^b ), \ \text{for all} \ x \in A
\end{displaymath}
has a solution $b \in A^1$.

\end{cor}

\begin{rmk}
In the Lagrangian Floer Homology setting, such element is called a \textit{bounding cochain} for Floer Homology.
\end{rmk}

As in the Lagrangian Floer Homology theory, we have some simple unobstructed cases in the geometrical realization:

1) if $a =0 \in I(A,m)$, then $MC(a)$ will be a solution set of the obstruction equation, such $(A,m)$ is called \textit{unobstructed}, 

2) if $I(A,m)$ contains some element $a \in A^2$ with the partial unital property,
then $MC(a)$ is also a solution set, such $(A,m)$ is called \textit{weakly unobstructed},

3) similarly, if we have element $b \in A^1$ with the partial unital property and:
\begin{displaymath}
 m_2(b,b)=0, \quad \ m_0(\textbf{1})+m_1(b) \in Z(m_2)
\end{displaymath}
then such $b$ is also a solution of the obstruction equation.\\

Notice that the Maurer-Cartan equation is non-linear, thus we have the following type result in some special cases:
\begin{prp} Given a curved $A_\infty$ algebra over field $\textbf{F}$ with characteristic 0, if $0 \in I(A,m)$ and $MC(0)$ admits a linear subset, i.e. there exists $b \in A^1$ such that $\lambda \cdot b \in MC(0)$ for all $\lambda \in \textbf{F}$, then $m_0(\textbf{1}) = 0$ and $m_k(b, \cdots, b)=0$ for all $k \geq 1$.
\end{prp}
\begin{rmk}
In the toric fano case \cite{6}, similar phenomena happens during the calculating of the weak bounding cochain.
\end{rmk}

We end up with an example to show the restriction of such almost identity deformation towards $\tilde{m}_0(1)=0$ as in the unobstructed case.

\textbf{Example.3} Given a smooth manifold $M$, let us consider the deformation of its de Rham algebra $(\Omega(M),d,\wedge)$. To fit the construction equation of $A_\infty$ algebra, we need some adjustment:
\begin{align*}
& A = \bigoplus \Omega^k (M)\\
m_1(x) = (-1)^{|x|} &d(x), \quad m_2(x,y)=(-1)^{|x|\cdot (|y|+1)}x \wedge y.
\end{align*}
Then $(A, m_1, m_2)$ becomes a strict $A_\infty$ algebra over field $\mathbf{R}$. Moreover, it is easy to check $(A,m_0,m_1,m_2)$ will be a curved $A_\infty$ algebra if and only if $m_0(1)$ is a closed 2-form. 

Let us take $(A,m_0,m_1,m_2)$ and $m_0(\textbf{1})=a$ a closed 2-form as the initial curved $A_\infty$ structure, and consider its deformation. Recall that given an almost endomorphism $f_0=b \in \Omega^1(M), f_1=\mathbf{I}$, we have the deformation given by:
\begin{align*}
&\tilde{m}_0(\textbf{1}) = m_0(\textbf{1})+m_1(b) = a-d(b),\\
&\tilde{m}_1(x) = m_1(x)+m_2(b,x)+m_2(x,b)= m_1(x)= (-1)^{|x|} d(x),\\
&\tilde{m}_2(x,y) = m_2(x,y)= (-1)^{|x|\cdot (|y|+1)}x \wedge y.
\end{align*}
Let us forget $\tilde{m}_1^2=0$ for a while, and consider the equation $\tilde{m}_0(\textbf{1})$=0. Then it is easy to see that we can achieve $\tilde{m}_0(\textbf{1})=0$ if and only if $m_0(\textbf{1})=a$ is an exact 2-form. Thus to realize $\tilde{m}_0(\textbf{1})=0$ in the rest case, i.e. $m_0(\textbf{1})$ is closed but not exact, we need generalized type of almost identity endomorphism for example as the \textit{bulk deformation} in the Lagrangian Floer theory. 

This also serves as an special example of the curved case, no matter what is the curvature term $m_0(\textbf{1})$ is, we always have $m_1^2=0$ for free. Under the almost endomorphism deformation, $\{m_k\}$ with $k \geq 1$ keeps invariant.

\newpage

At last, let us discuss the general case of such deformation. It is easy to check that if we replace $f_1$ by other automorphism on the vector space $A$, then the deformed $A_\infty$ structure $\tilde{m}$ varies, however the obstruction equation is essentially the same. More generally we could loose the condition and consider the deformation given by a general endomorphism which is not necessary weakly strict anymore. 

For simplicity, we still take $f_1=\mathbf{I}$ as an example. Let us consider the construction equation for the endomorphism maps in this case. First, we have the same Maurer-Cartan equation: 
\begin{displaymath}
\tilde m_0(\textbf{1}) = m(e^b)
\end{displaymath}
Let us denote $\tilde m_0(\textbf{1})$ by $a$ as before. Then $\tilde{m}_1$ and $\tilde{m}_2$ will goes to:
\begin{align*}
\tilde{m}_1(x) = & m (e^b x e^b) +(-1)^{|x|} f_2(x, a) - f_2(a,x)\\
\tilde{m}_2(x,y) = & m (e^b x  e^b y e^b) + m(e^b  f_2(x,y) e^b) + [(-1)^{|x|}f_2(x, \tilde{m}_1(y)) - f_2(\tilde{m}_1(x),y )] \\
              & - [f_3(a,x,y)- (-1)^{|x|}f_3(x,a,y) + (-1)^{|x|+|y|}f_3(x,y,a)]                        
\end{align*}
Notice that $\tilde{m}_1(a) =0$, then the obstruction for $\tilde{m}^2_1 = 0$ goes to:
\begin{align*}
& m(e^bae^bxe^b)+m(e^bf_2(a,x)e^b) + f_2(a,\tilde{m}_1(x))\\
 = & (-1)^{|x|} [m(e^bxe^bae^b)+m(e^bf_2(x,a)e^b)-f_2(\tilde{m}_1(x),a)]
\end{align*}

Notice that only $f_1, f_2$ components of the endomorphism are involved in the obstruction equation even in the general case. Formally more freedom will be left to $f_2$ to achieve the condition $\tilde{m}_1^2 =0$ for cohomology consideration. In the geometrical realization, if we further have the cyclic structure on $(A,m)$, then calculation will be simplified.

In the special case, when $a$ has the center property with respect to $f_2$, i.e. $f_2(a,x) = (-1)^{|x|}f_2(x,a)$, then the obstruction equation returns to the same one as in the almost identity endomorphism case. Moreover, it is easy to check that $a=0, b \in MC(a)$ is still a solution set for the obstruction equation. To find other solutions, we need further understanding of the $f_2$ component in the geometrical realization. In the Lagrangian Floer Homology case, we expect it will be clarified by future study of \textit{Lagrangian cobordism} \cite{1}.
\\
\\

\textbf{Acknowledgements.} This work was supported by Institute for Basic Science (IBS). The author would like to thank the IBS center for Geometry and Physics in Korea for providing financial supports and excellent environments for research. The author also would like to thank his advisor to give an inspiring course about $A_\infty$ algebra and Symplectic Algebraic Topology in UW-Madison in America.

\end{document}